\documentclass[global]{svjour}
\usepackage{amsfonts}
\usepackage{amsmath}

\input{tcilatex}
\begin{document}

\journalname{Name Your Journal}
\title{A Picard-S Iterative Scheme for Approximating Fixed Point of
Weak-Contraction Mappings}
\author{Faik G\"{U}RSOY}
\institute{Department of Mathematics, Yildiz Technical University, Davutpasa Campus,
Esenler, 34220 Istanbul, Turkey 
\email{faikgursoy02@hotmail.com}%
}
\offprints{Offprints Assistant}
\mail{Address for offprint requests}
\maketitle

\begin{abstract}
We study the convergence analysis of a Picard-S iteration method for a
particular class of weak-contraction mappings. Furthermore, we prove a data
dependence result for fixed point of the class of weak-contraction mappings
with the help of the Picard-S iteration methods.
\end{abstract}


\section{Introduction}

Most of the problems that arise in different disciplines of science can be
formulated by the equations in the form%
\begin{equation}
Fx=0\text{,}  \label{eqn1}
\end{equation}%
where $F$ is some function. The equations given by (1) can easily be
reformulated as the fixed point equations of type 
\begin{equation}
Tx=x\text{.}  \label{eqn2}
\end{equation}%
where $T$ is a self-map of an ambient space $X$ and $x\in X$. These
equations are often classified as linear or nonlinear, depending on whether
the mappings used in the equation is linear with respect to the variables.
Over the years, a considerable attention has been paid to solving such
equations by using different techniques such as direct and iterative
methods. In case of linear equations, both direct and iterative methods are
used to obtain solutions of the equations. But in case of nonlinear
equations, due to various reasons, direct methods can be impractical or fail
in solving equations, and thus iterative methods become a viable
alternative. Nonlinear problems are of importance interest to
mathematicians, physicists and engineers and many other scientists, simply
because most systems are intrinsically nonlinear in nature. That is why,
researchers in various disciplines of sciences are often faced with the
solving such problems. It would be hard to fudge that the role of iterative
approximation of fixed points have played in the recent progress of
nonlinear science. Indeed, the phrase iterative approximation has been
introduced to describe repetition-based researches into nonlinear problems
that are inaccessible to analytic methods. For this reason, the iterative
approximation of fixed points has become one of the major and basic tools in
the theory of equations, and as a result, numerous iterative methods have
been introduced or improved and studied for many years in detail from
various points of aspects by a wide audience of researchers, see, [1-14,
16-42, 44, 45].

In this paper, we show that a Picard-S iteration method \cite{Gursoy} can be
used to approximate fixed point of weak-contraction mappings. Also, we show
that this iteration method is equivalent and converges faster than CR
iteration method \cite{CR} for the aforementioned class of mappings.
Furthermore, by providing an example, it is shown that the Picard-S
iteration method converges faster than CR iteration method and hence also
faster than all Picard \cite{Picard}, Mann \cite{Mann}, Ishikawa \cite%
{Ishikawa}, Noor \cite{Noor}, SP \cite{SP}, S \cite{S} and some other
iteration methods in the existing literature when applied to
weak-contraction mappings. Finally, a data dependence result is proven for
fixed point of weak-contraction mappings with the help of the Picard-S
iteration method.

Throughout this paper the set of all positive integers and zero is shown by $%
\mathbb{N}
$. Let $B$ be a Banach space, $D$ be a nonempty closed convex subset of $B$
and $T$ a self-map of $D$. An element $x_{\ast }$ of $D$ is called a fixed
point of $T$ if and only if $Tx_{\ast }=x_{\ast }$. The set of all fixed
point of $T$ denoted by $F_{T}$. Let $\left\{ a_{n}^{i}\right\}
_{n=0}^{\infty }$, $i\in \left\{ 0,1,2\right\} $ be real sequences in $\left[
0,1\right] $ satisfying certain control condition(s).

Renowned Picard iteration method \cite{Picard} is formulated as follow%
\begin{equation}
\left\{ 
\begin{array}{c}
p_{0}\in D\text{, \ \ \ \ \ \ \ \ \ \ \ \ \ \ \ \ \ } \\ 
p_{n+1}=Tp_{n}\text{, }n\in 
\mathbb{N}
\text{,}%
\end{array}%
\right.   \label{eqn3}
\end{equation}%
and generally used to approximate fixed points of contraction mappings
satisfying: for all $x$, $y\in B$ there exists a $\delta \in \left(
0,1\right) $ such that%
\begin{equation}
\left\Vert Tx-Ty\right\Vert \leq \delta \left\Vert x-y\right\Vert \text{.}
\label{eqn4}
\end{equation}%
The following iteration methods are known as Noor \cite{Noor} and SP \cite%
{SP} iterations, respectively:%
\begin{equation}
\left\{ 
\begin{array}{c}
\omega _{0}\in D\text{, \ \ \ \ \ \ \ \ \ \ \ \ \ \ \ \ \ \ \ \ \ \ \ \ \ \
\ \ \ \ \ } \\ 
\omega _{n+1}=\left( 1-a_{n}^{0}\right) \omega _{n}+a_{n}^{0}T\varpi _{n}%
\text{, \ \ } \\ 
\varpi _{n}=\left( 1-a_{n}^{1}\right) \omega _{n}+a_{n}^{1}T\rho _{n}\text{, 
} \\ 
\rho _{n}=\left( 1-a_{n}^{2}\right) \omega _{n}+a_{n}^{2}T\omega _{n}\text{, 
}n\in 
\mathbb{N}
\text{,}%
\end{array}%
\right.   \label{eqn5}
\end{equation}%
\begin{equation}
\left\{ 
\begin{array}{c}
q_{0}\in D\text{,\ \ \ \ \ \ \ \ \ \ \ \ \ \ \ \ \ \ \ \ \ \ \ \ \ \ \ \ \ \
\ } \\ 
q_{n+1}=\left( 1-a_{n}^{0}\right) r_{n}+a_{n}^{0}Tr_{n}\text{, \ \ } \\ 
r_{n}=\left( 1-a_{n}^{1}\right) s_{n}+a_{n}^{1}Ts_{n}\text{,} \\ 
s_{n}=\left( 1-a_{n}^{2}\right) q_{n}+a_{n}^{2}Tq_{n}\text{, }n\in 
\mathbb{N}
\text{,}%
\end{array}%
\right.   \label{eqn6}
\end{equation}

\begin{remark}
(i) If $a_{n}^{2}=0$ for each $n\in 
\mathbb{N}
$ , then the Noor iteration method reduces to iterative method of Ishikawa 
\cite{Ishikawa}.

(ii) If $a_{n}^{2}=0$ for each $n\in 
\mathbb{N}
$ , then the SP iteration method reduces to iterative method of Thianwan 
\cite{iam}. 

(iii) When $a_{n}^{1}=$ $a_{n}^{2}=0$ for each $n\in 
\mathbb{N}
$, then both Noor and SP iteration methods reduce to an iteration method due
to Mann \cite{Mann}.  
\end{remark}

Recently, G\"{u}rsoy and Karakaya \cite{Gursoy} introduced a Picard-S
iterative scheme as follows: 
\begin{equation}
\left\{ 
\begin{array}{c}
x_{0}\in D\text{, \ \ \ \ \ \ \ \ \ \ \ \ \ \ \ \ \ \ \ \ \ \ \ \ \ \ \ \ \
\ \ \ \ \ \ \ \ \ \ \ } \\ 
x_{n+1}=Ty_{n}\text{, \ \ \ \ \ \ \ \ \ \ \ \ \ \ \ \ \ \ \ \ \ \ \ \ \ \ \
\ \ \ \ \ \ \ } \\ 
y_{n}=\left( 1-a_{n}^{1}\right) Tx_{n}+a_{n}^{1}Tz_{n}\text{, \ \ \ \ \ \ }
\\ 
z_{n}=\left( 1-a_{n}^{2}\right) x_{n}+a_{n}^{2}Tx_{n}\text{, }n\in 
\mathbb{N}
\text{,}%
\end{array}%
\right.   \label{eqn7}
\end{equation}%
The following definitions and lemmas will be needed in obtaining the main
results of this article.

\begin{definition}
\cite{Berinde} Let $\left\{ a_{n}\right\} _{n=0}^{\infty }$ and $\left\{
b_{n}\right\} _{n=0}^{\infty }$ be two sequences of real numbers with limits 
$a$ and $b$, respectively. Assume that there exists%
\begin{equation}
\underset{n\rightarrow \infty }{\lim }\frac{\left\vert a_{n}-a\right\vert }{%
\left\vert b_{n}-b\right\vert }=l\text{.}  \label{eqn8}
\end{equation}

(i) If $l=0$, the we say that $\left\{ a_{n}\right\} _{n=0}^{\infty }$
converges faster to $a$ than $\left\{ b_{n}\right\} _{n=0}^{\infty }$ to $b$.

(ii) If $0<l<\infty $, then we say that $\left\{ a_{n}\right\}
_{n=0}^{\infty }$ and $\left\{ b_{n}\right\} _{n=0}^{\infty }$ have the same
rate of convergence.
\end{definition}

\begin{definition}
\cite{Berinde} Assume that for two fixed point iteration processes $\left\{
u_{n}\right\} _{n=0}^{\infty }$ and $\left\{ v_{n}\right\} _{n=0}^{\infty }$
both converging to the same fixed point $p$, the following error predictions%
\begin{equation}
\left\Vert u_{n}-p\right\Vert \leq a_{n}\text{ for all }n\in 
\mathbb{N}
\text{,}  \label{eqn9}
\end{equation}%
\begin{equation}
\left\Vert v_{n}-p\right\Vert \leq b_{n}\text{ for all }n\in 
\mathbb{N}
\text{,}  \label{eqn10}
\end{equation}%
are available where $\left\{ a_{n}\right\} _{n=0}^{\infty }$ and $\left\{
b_{n}\right\} _{n=0}^{\infty }$ are two sequences of positive numbers
(converging to zero). If $\left\{ a_{n}\right\} _{n=0}^{\infty }$ converges
faster than $\left\{ b_{n}\right\} _{n=0}^{\infty }$, then $\left\{
u_{n}\right\} _{n=0}^{\infty }$ converges faster than $\left\{ v_{n}\right\}
_{n=0}^{\infty }$ to $p$.
\end{definition}

\begin{definition}
\cite{Vasile} Let $\left( B,\left\Vert \cdot \right\Vert \right) $ be a
Banach space. A map $T:B\rightarrow B$ is called weak-contraction if there
exist a constant $\delta \in \left( 0,1\right) $ and some $L\geq 0$ such that%
\begin{equation}
\left\Vert Tx-Ty\right\Vert \leq \delta \left\Vert x-y\right\Vert
+L\left\Vert y-Ty\right\Vert \text{, for all }x\text{, }y\in B\text{.}
\label{eqn11}
\end{equation}
\end{definition}

\begin{definition}
\cite{Vasile} Let $T$,$\widetilde{T}:B\rightarrow B$ be two operators. We
say that $\widetilde{T}$ is an approximate operator of $T$ if for all $x\in B
$ and for a fixed $\varepsilon >0$ we have 
\begin{equation}
\left\Vert Tx-\widetilde{T}x\right\Vert \leq \varepsilon .  \label{eqn12}
\end{equation}
\end{definition}

\begin{lemma}
\cite{Weng}Let $\left\{ \beta _{n}\right\} _{n=0}^{\infty }$ and $\left\{
\rho _{n}\right\} _{n=0}^{\infty }$ be nonnegative real sequences satisfying
the following inequality:%
\begin{equation}
\beta _{n+1}\leq \left( 1-\lambda _{n}\right) \beta _{n}+\rho _{n}\text{,}
\label{eqn13}
\end{equation}%
where $\lambda _{n}\in \left( 0,1\right) $, for all $n\geq n_{0}$, $%
\dsum\nolimits_{n=1}^{\infty }\lambda _{n}=\infty $, and $\frac{\rho _{n}}{%
\lambda _{n}}\rightarrow 0$ as $n\rightarrow \infty $. Then $%
\lim_{n\rightarrow \infty }\beta _{n}=0$.
\end{lemma}

\begin{lemma}
\cite{Data Is 2} Let $\left\{ \beta _{n}\right\} _{n=0}^{\infty }$ be a
nonnegative sequence for which one assumes there exists $n_{0}\in 
\mathbb{N}
$, such that for all $n\geq n_{0}$ one has satisfied the inequality 
\begin{equation}
\beta _{n+1}\leq \left( 1-\mu _{n}\right) \beta _{n}+\mu _{n}\gamma _{n}%
\text{,}  \label{eqn15}
\end{equation}%
where $\mu _{n}\in \left( 0,1\right) ,$ for all $n\in 
\mathbb{N}
$, $\sum\limits_{n=0}^{\infty }\mu _{n}=\infty $ and $\gamma _{n}\geq 0$, $%
\forall n\in 
\mathbb{N}
$. Then the following inequality holds 
\begin{equation}
0\leq \lim \sup_{n\rightarrow \infty }\beta _{n}\leq \lim \sup_{n\rightarrow
\infty }\gamma _{n}.  \label{eqn16}
\end{equation}
\end{lemma}

\section{Main Results}

\begin{theorem}
Let $T:D\rightarrow D$ be a weak-contraction map satisfying condition (11)
with $F_{T}\neq \emptyset $ and $\left\{ x_{n}\right\} _{n=0}^{\infty }$ an
iterative sequence defined by (7) with real sequences $\left\{
a_{n}^{i}\right\} _{n=0}^{\infty }$, $i\in \left\{ 1,2\right\} $ in $\left[
0,1\right] $ satisfying $\sum_{k=0}^{\infty }a_{k}^{1}a_{k}^{2}=\infty $.
Then $\ \left\{ x_{n}\right\} _{n=0}^{\infty }$ converges to a unique fixed
point $u^{\ast }$of $T$.
\end{theorem}

\begin{proof}
Uniqueness of $u^{\ast }$ comes from condition (11). Using Picard-S
iterative scheme (7) and condition (11), we obtain%
\begin{eqnarray}
\left\Vert z_{n}-u^{\ast }\right\Vert  &\leq &\left( 1-a_{n}^{2}\right)
\left\Vert x_{n}-u^{\ast }\right\Vert +a_{n}^{2}\left\Vert Tx_{n}-Tu^{\ast
}\right\Vert   \notag \\
&\leq &\left( 1-a_{n}^{2}\right) \left\Vert x_{n}-u^{\ast }\right\Vert
+a_{n}^{2}\delta \left\Vert x_{n}-u^{\ast }\right\Vert +a_{n}^{2}L\left\Vert
u^{\ast }-Tu^{\ast }\right\Vert   \notag \\
&=&\left[ 1-a_{n}^{2}\left( 1-\delta \right) \right] \left\Vert
x_{n}-u^{\ast }\right\Vert \text{,}  \label{eqn17}
\end{eqnarray}%
\begin{eqnarray}
\left\Vert y_{n}-u^{\ast }\right\Vert  &\leq &\left( 1-a_{n}^{1}\right)
\left\Vert Tx_{n}-Tu^{\ast }\right\Vert +a_{n}^{1}\left\Vert Tz_{n}-Tu^{\ast
}\right\Vert   \notag \\
&\leq &\left( 1-a_{n}^{1}\right) \delta \left\Vert x_{n}-u^{\ast
}\right\Vert +a_{n}^{1}\delta \left\Vert z_{n}-u^{\ast }\right\Vert \text{,}
\label{eqn18}
\end{eqnarray}%
\begin{equation}
\left\Vert x_{n+1}-u^{\ast }\right\Vert \leq \delta \left\Vert y_{n}-u^{\ast
}\right\Vert \text{.}  \label{eqn19}
\end{equation}%
Combining (16), (17) and (18)%
\begin{equation}
\left\Vert x_{n+1}-u^{\ast }\right\Vert \leq \delta ^{2}\left[
1-a_{n}^{1}a_{n}^{2}\left( 1-\delta \right) \right] \left\Vert x_{n}-u^{\ast
}\right\Vert \text{.}  \label{eqn20}
\end{equation}%
By induction%
\begin{eqnarray}
\left\Vert x_{n+1}-u^{\ast }\right\Vert  &\leq &\delta ^{2\left( n+1\right)
}\dprod\nolimits_{k=0}^{n}\left[ 1-a_{k}^{1}a_{k}^{2}\left( 1-\delta \right) %
\right] \left\Vert x_{0}-u^{\ast }\right\Vert   \notag \\
&\leq &\delta ^{2\left( n+1\right) }\left\Vert x_{0}-u^{\ast }\right\Vert
^{n+1}e^{-\left( 1-\delta \right) \dsum\nolimits_{k=0}^{n}a_{k}^{1}a_{k}^{2}}%
\text{.}  \label{eqn21}
\end{eqnarray}%
Since $\sum_{k=0}^{\infty }a_{k}^{1}a_{k}^{2}=\infty $,%
\begin{equation}
e^{-\left( 1-\delta \right)
\dsum\nolimits_{k=0}^{n}a_{k}^{1}a_{k}^{2}}\rightarrow 0\text{ as }%
n\rightarrow \infty \text{,}  \label{eqn22}
\end{equation}%
which implies $\lim_{n\rightarrow \infty }\left\Vert x_{n}-u^{\ast
}\right\Vert $.
\end{proof}

\begin{theorem}
Let $T:D\rightarrow D$ with fixed point $u^{\ast }\in F_{T}\neq \emptyset $
be as in Theorem 1 and $\{q_{n}\}_{n=0}^{\infty }$, $\{x_{n}\}_{n=0}^{\infty
}$ two iterative sequences defined by SP (6) and Picard-S (7) iteration
methods with real sequences $\left\{ a_{n}^{i}\right\} _{n=0}^{\infty }$, $%
i\in \left\{ 0,1,2\right\} $ in $\left[ 0,1\right] $ satisfying $%
\sum_{k=0}^{n}a_{k}^{1}a_{k}^{2}=\infty $. Then the following are equivalent:

(i) $\lim_{n\rightarrow \infty }\left\Vert x_{n}-u^{\ast }\right\Vert =0$;

(ii) $\lim_{n\rightarrow \infty }\left\Vert q_{n}-u^{\ast }\right\Vert =0$.

\begin{proof}
(i)$\Rightarrow $(ii): It follows from (6), (7), and condition (11) that%
\begin{eqnarray}
\left\Vert x_{n+1}-q_{n+1}\right\Vert  &=&\left\Vert \left(
1-a_{n}^{0}\right) \left( Ty_{n}-r_{n}\right) +a_{n}^{0}\left(
Ty_{n}-Tr_{n}\right) \right\Vert   \label{eqn23} \\
&\leq &\left( 1-a_{n}^{0}\right) \left\Vert Ty_{n}-r_{n}\right\Vert
+a_{n}^{0}\left\Vert Ty_{n}-Tr_{n}\right\Vert   \notag \\
&\leq &\left[ 1-a_{n}^{0}\left( 1-\delta \right) \right] \left\Vert
y_{n}-r_{n}\right\Vert +\left[ 1-a_{n}^{0}\left( 1-L\right) \right]
\left\Vert y_{n}-Ty_{n}\right\Vert \text{,}  \notag
\end{eqnarray}%
\begin{eqnarray}
\left\Vert y_{n}-r_{n}\right\Vert  &=&\left\Vert \left( 1-a_{n}^{1}\right)
\left( Tx_{n}-s_{n}\right) +a_{n}^{1}\left( Tz_{n}-Ts_{n}\right) \right\Vert 
\label{eqn24} \\
&\leq &\left( 1-a_{n}^{1}\right) \left\Vert Tx_{n}-s_{n}\right\Vert
+a_{n}^{1}\left\Vert Tz_{n}-Ts_{n}\right\Vert   \notag \\
&\leq &\left( 1-a_{n}^{1}\right) \left\Vert Tx_{n}-s_{n}\right\Vert
+a_{n}^{1}\delta \left\Vert z_{n}-s_{n}\right\Vert +a_{n}^{1}L\left\Vert
z_{n}-Tz_{n}\right\Vert \text{,}  \notag
\end{eqnarray}%
\begin{eqnarray}
\left\Vert Tx_{n}-s_{n}\right\Vert  &=&\left\Vert \left( 1-a_{n}^{2}\right)
\left( Tx_{n}-q_{n}\right) +a_{n}^{2}\left( Tx_{n}-Tq_{n}\right) \right\Vert 
\label{eqn25} \\
&\leq &\left[ 1-a_{n}^{2}\left( 1-\delta \right) \right] \left\Vert
x_{n}-q_{n}\right\Vert +\left[ 1-a_{n}^{2}\left( 1-L\right) \right]
\left\Vert x_{n}-Tx_{n}\right\Vert \text{,}  \notag
\end{eqnarray}%
\begin{eqnarray}
\left\Vert z_{n}-s_{n}\right\Vert  &\leq &\left( 1-a_{n}^{2}\right)
\left\Vert x_{n}-q_{n}\right\Vert +a_{n}^{2}\left\Vert
Tx_{n}-Tq_{n}\right\Vert   \label{eqn26} \\
&\leq &\left[ 1-a_{n}^{2}\left( 1-\delta \right) \right] \left\Vert
x_{n}-q_{n}\right\Vert +a_{n}^{2}L\left\Vert x_{n}-Tx_{n}\right\Vert \text{.}
\notag
\end{eqnarray}%
Combining (22), (23), (24), and (25)%
\begin{eqnarray}
\left\Vert x_{n+1}-q_{n+1}\right\Vert  &\leq &\left[ 1-a_{n}^{0}\left(
1-\delta \right) \right] \left[ 1-a_{n}^{1}\left( 1-\delta \right) \right] %
\left[ 1-a_{n}^{2}\left( 1-\delta \right) \right] \left\Vert
x_{n}-q_{n}\right\Vert   \label{eqn27} \\
&&+\left[ 1-a_{n}^{0}\left( 1-\delta \right) \right] \left\{ \left(
1-a_{n}^{1}\right) \left[ 1-a_{n}^{2}\left( 1-L\right) \right]
+a_{n}^{1}a_{n}^{2}\delta L\right\} \left\Vert x_{n}-Tx_{n}\right\Vert  
\notag \\
&&+\left[ 1-a_{n}^{0}\left( 1-\delta \right) \right] a_{n}^{1}L\left\Vert
z_{n}-Tz_{n}\right\Vert +\left[ 1-a_{n}^{0}\left( 1-L\right) \right]
\left\Vert y_{n}-Ty_{n}\right\Vert \text{.}  \notag
\end{eqnarray}%
It follows from the facts $\delta \in \left( 0,1\right) $ and $a_{n}^{i}\in %
\left[ 0,1\right] $, $\forall n\in 
\mathbb{N}
$, $i\in \left\{ 0,1,2\right\} $ that%
\begin{equation}
\left[ 1-a_{n}^{0}\left( 1-\delta \right) \right] \left[ 1-a_{n}^{1}\left(
1-\delta \right) \right] \left[ 1-a_{n}^{2}\left( 1-\delta \right) \right]
<1-a_{n}^{1}a_{n}^{2}\left( 1-\delta \right) \text{.}  \label{eqn28}
\end{equation}%
Hence, inequality (26) becomes%
\begin{eqnarray}
\left\Vert x_{n+1}-q_{n+1}\right\Vert  &\leq &\left[ 1-a_{n}^{1}a_{n}^{2}%
\left( 1-\delta \right) \right] \left\Vert x_{n}-q_{n}\right\Vert 
\label{eqn29} \\
&&+\left[ 1-a_{n}^{0}\left( 1-\delta \right) \right] \left\{ \left(
1-a_{n}^{1}\right) \left[ 1-a_{n}^{2}\left( 1-L\right) \right]
+a_{n}^{1}a_{n}^{2}\delta L\right\} \left\Vert x_{n}-Tx_{n}\right\Vert  
\notag \\
&&+\left[ 1-a_{n}^{0}\left( 1-\delta \right) \right] a_{n}^{1}L\left\Vert
z_{n}-Tz_{n}\right\Vert +\left[ 1-a_{n}^{0}\left( 1-L\right) \right]
\left\Vert y_{n}-Ty_{n}\right\Vert \text{.}  \notag
\end{eqnarray}%
Denote that%
\begin{eqnarray}
\beta _{n} &:&=\left\Vert x_{n}-q_{n}\right\Vert \text{, \ \ }  \label{eqn30}
\\
\lambda _{n} &:&=a_{n}^{1}a_{n}^{2}\left( 1-\delta \right) \in \left(
0,1\right) \text{, \ }  \notag \\
\rho _{n} &:&=\left[ 1-a_{n}^{0}\left( 1-\delta \right) \right] \left\{
\left( 1-a_{n}^{1}\right) \left[ 1-a_{n}^{2}\left( 1-L\right) \right]
+a_{n}^{1}a_{n}^{2}\delta L\right\} \left\Vert x_{n}-Tx_{n}\right\Vert  
\notag \\
&&+\left[ 1-a_{n}^{0}\left( 1-\delta \right) \right] a_{n}^{1}L\left\Vert
z_{n}-Tz_{n}\right\Vert +\left[ 1-a_{n}^{0}\left( 1-L\right) \right]
\left\Vert y_{n}-Ty_{n}\right\Vert \text{.}  \notag
\end{eqnarray}%
Since $\lim_{n\rightarrow \infty }\left\Vert x_{n}-u^{\ast }\right\Vert =0$
and $Tu^{\ast }=u^{\ast }$%
\begin{equation}
\lim_{n\rightarrow \infty }\left\Vert x_{n}-Tx_{n}\right\Vert
=\lim_{n\rightarrow \infty }\left\Vert y_{n}-Ty_{n}\right\Vert
=\lim_{n\rightarrow \infty }\left\Vert z_{n}-Tz_{n}\right\Vert =0\text{,}
\label{eqn31}
\end{equation}%
which implies $\frac{\rho _{n}}{\lambda _{n}}\rightarrow 0$ as $n\rightarrow
\infty $. Therefore, inequality (28) perform all assumptions in Lemma 1 and
thus we obtain $\lim_{n\rightarrow \infty }\left\Vert x_{n}-q_{n}\right\Vert
=0$. Since%
\begin{equation}
\left\Vert q_{n}-u^{\ast }\right\Vert \leq \left\Vert x_{n}-q_{n}\right\Vert
+\left\Vert x_{n}-u^{\ast }\right\Vert \rightarrow 0\text{ as }n\rightarrow
\infty \text{,}  \label{eqn32}
\end{equation}%
$\lim_{n\rightarrow \infty }\left\Vert q_{n}-u^{\ast }\right\Vert =0$.

(ii)$\Rightarrow $(i): It follows from (6), (7), and condition (11) that%
\begin{eqnarray}
\left\Vert q_{n+1}-x_{n+1}\right\Vert  &=&\left\Vert
r_{n}-Ty_{n}+a_{n}^{0}\left( Tr_{n}-r_{n}\right) \right\Vert   \label{eqn33}
\\
&\leq &\delta \left\Vert r_{n}-y_{n}\right\Vert +\left( 1+a_{n}^{0}+L\right)
\left\Vert r_{n}-Tr_{n}\right\Vert \text{,}  \notag
\end{eqnarray}%
\begin{eqnarray}
\left\Vert r_{n}-y_{n}\right\Vert  &\leq &\left( 1-a_{n}^{1}\right)
\left\Vert s_{n}-Tx_{n}\right\Vert +a_{n}^{1}\left\Vert
Ts_{n}-Tz_{n}\right\Vert   \label{eqn34} \\
&\leq &\left( 1-a_{n}^{1}\right) \left\Vert s_{n}-Tx_{n}\right\Vert
+a_{n}^{1}\delta \left\Vert s_{n}-z_{n}\right\Vert +a_{n}^{1}L\left\Vert
s_{n}-Ts_{n}\right\Vert \text{,}  \notag
\end{eqnarray}%
\begin{eqnarray}
\left\Vert s_{n}-Tx_{n}\right\Vert  &\leq &\left\Vert
Ts_{n}-Tx_{n}\right\Vert +\left\Vert s_{n}-Ts_{n}\right\Vert   \label{eqn35}
\\
&\leq &\delta \left\Vert s_{n}-x_{n}\right\Vert +\left( 1+L\right)
\left\Vert s_{n}-Ts_{n}\right\Vert   \notag \\
&\leq &\delta \left\Vert q_{n}-x_{n}\right\Vert +\delta a_{n}^{2}\left\Vert
Tq_{n}-q_{n}\right\Vert +\left( 1+L\right) \left\Vert
s_{n}-Ts_{n}\right\Vert \text{,}  \notag
\end{eqnarray}%
\begin{eqnarray}
\left\Vert s_{n}-z_{n}\right\Vert  &\leq &\left( 1-a_{n}^{2}\right)
\left\Vert q_{n}-x_{n}\right\Vert +a_{n}^{2}\left\Vert
Tq_{n}-Tx_{n}\right\Vert   \label{eqn36} \\
&\leq &\left[ 1-a_{n}^{2}\left( 1-\delta \right) \right] \left\Vert
q_{n}-x_{n}\right\Vert +a_{n}^{2}L\left\Vert q_{n}-Tq_{n}\right\Vert \text{.}
\notag
\end{eqnarray}%
Combining (32), (33), (34), and (35)%
\begin{eqnarray}
\left\Vert q_{n+1}-x_{n+1}\right\Vert  &\leq &\delta ^{2}\left[
1-a_{n}^{1}a_{n}^{2}\left( 1-\delta \right) \right] \left\Vert
q_{n}-x_{n}\right\Vert   \label{eqn37} \\
&&+\delta ^{2}a_{n}^{2}\left[ 1-a_{n}^{1}\left( 1-L\right) \right]
\left\Vert q_{n}-Tq_{n}\right\Vert   \notag \\
&&+\left( 1+a_{n}^{0}+L\right) \left\Vert r_{n}-Tr_{n}\right\Vert +\delta
\left( 1-a_{n}^{1}+L\right) \left\Vert s_{n}-Ts_{n}\right\Vert \text{.} 
\notag
\end{eqnarray}%
Since $\delta \in \left( 0,1\right) $%
\begin{equation}
\delta ^{2}\left[ 1-a_{n}^{1}a_{n}^{2}\left( 1-\delta \right) \right]
<1-a_{n}^{1}a_{n}^{2}\left( 1-\delta \right) \text{.}  \label{eqn38}
\end{equation}%
Hence, inequality (36) becomes%
\begin{eqnarray}
\left\Vert q_{n+1}-x_{n+1}\right\Vert  &\leq &\left[ 1-a_{n}^{1}a_{n}^{2}%
\left( 1-\delta \right) \right] \left\Vert q_{n}-x_{n}\right\Vert 
\label{eqn39} \\
&&+\delta ^{2}a_{n}^{2}\left[ 1-a_{n}^{1}\left( 1-L\right) \right]
\left\Vert q_{n}-Tq_{n}\right\Vert   \notag \\
&&+\left( 1+a_{n}^{0}+L\right) \left\Vert r_{n}-Tr_{n}\right\Vert +\delta
\left( 1-a_{n}^{1}+L\right) \left\Vert s_{n}-Ts_{n}\right\Vert \text{.} 
\notag
\end{eqnarray}%
Denote that%
\begin{eqnarray}
\beta _{n} &:&=\left\Vert q_{n}-x_{n}\right\Vert \text{, \ \ }  \label{eqn40}
\\
\lambda _{n} &:&=a_{n}^{1}a_{n}^{2}\left( 1-\delta \right) \in \left(
0,1\right) \text{, \ }  \notag \\
\rho _{n} &:&=\delta ^{2}a_{n}^{2}\left[ 1-a_{n}^{1}\left( 1-L\right) \right]
\left\Vert q_{n}-Tq_{n}\right\Vert   \notag \\
&&+\left( 1+a_{n}^{0}+L\right) \left\Vert r_{n}-Tr_{n}\right\Vert +\delta
\left( 1-a_{n}^{1}+L\right) \left\Vert s_{n}-Ts_{n}\right\Vert \text{.} 
\notag
\end{eqnarray}%
Since $\lim_{n\rightarrow \infty }\left\Vert q_{n}-u^{\ast }\right\Vert =0$
and $Tu^{\ast }=u^{\ast }$%
\begin{equation}
\lim_{n\rightarrow \infty }\left\Vert q_{n}-Tq_{n}\right\Vert
=\lim_{n\rightarrow \infty }\left\Vert r_{n}-Tr_{n}\right\Vert
=\lim_{n\rightarrow \infty }\left\Vert s_{n}-Ts_{n}\right\Vert =0\text{,}
\label{eqn41}
\end{equation}%
which implies $\frac{\rho _{n}}{\lambda _{n}}\rightarrow 0$ as $n\rightarrow
\infty $. Therefore, inequality (38) perform all assumptions in Lemma 1 and
thus we obtain $\lim_{n\rightarrow \infty }\left\Vert q_{n}-x_{n}\right\Vert
=0$. Since%
\begin{equation}
\left\Vert x_{n}-u^{\ast }\right\Vert \leq \left\Vert q_{n}-x_{n}\right\Vert
+\left\Vert q_{n}-u^{\ast }\right\Vert \rightarrow 0\text{ as }n\rightarrow
\infty \text{,}  \label{eqn42}
\end{equation}%
$\lim_{n\rightarrow \infty }\left\Vert x_{n}-u^{\ast }\right\Vert =0$.
\end{proof}
\end{theorem}

Taking R. Chugh et al.'s result (\cite{CR}, Corollary 3.2) into account,
Theorem 2 leads to the following corollary under weaker assumption:

\begin{corollary}
Let $T:D\rightarrow D$ with fixed point $u^{\ast }\in F_{T}\neq \emptyset $
be as in Theorem 1. Then the followings are equivalent:

1)The Picard iteration method (3) converges to $u^{\ast }$,

2) The Mann iteration method \cite{Mann} converges to $u^{\ast }$,

3) The Ishikawa iteration method \cite{Ishikawa} converges to $u^{\ast }$,

4) The Noor iteration method (5) converges to $u^{\ast }$,

5) S-iteration method \cite{S} converges to $u^{\ast }$,

6) The SP-iteration method (6) converges to $u^{\ast }$,

7) CR-iteration method \cite{CR} converges to $u^{\ast }$,

8) The Picard-S iteration method (7) converges to $u^{\ast }$.
\end{corollary}

\begin{theorem}
Let $T:D\rightarrow D$ with fixed point $u^{\ast }\in F_{T}\neq \emptyset $
be as in Theorem 1. Suppose that $\left\{ \omega _{n}\right\} _{n=0}^{\infty
}$, $\left\{ q_{n}\right\} _{n=0}^{\infty }$ and $\left\{ x_{n}\right\}
_{n=0}^{\infty }$ are iterative sequences, respectively, defined by Noor
(5), SP (6) and Picard-S (7) iterative schemes with real sequences $\left\{
a_{n}^{i}\right\} _{n=0}^{\infty }\subset \left[ 0,1\right] $, $i\in \left\{
0\text{,}1\text{,}2\right\} $ satisfying

(i) $0\leq a_{n}^{i}<\frac{1}{1+\delta }$,

(ii) $\lim_{n\rightarrow \infty }a_{n}^{i}=0$.

Then the iterative sequence defined by (7) converges faster than the
iterative sequences defined by (5) and (6) to a unique fixed point of $T$,
provided that the initial point is the same for all iterations.
\end{theorem}

\begin{proof}
From inequality (20), we have%
\begin{equation}
\left\Vert x_{n+1}-u^{\ast }\right\Vert \leq \delta ^{2\left( n+1\right)
}\left\Vert x_{0}-u^{\ast }\right\Vert ^{n+1}\dprod\nolimits_{k=0}^{n}\left[
1-a_{k}^{1}a_{k}^{2}\left( 1-\delta \right) \right]   \label{eqn43}
\end{equation}%
Using (6) we obtain%
\begin{eqnarray}
\left\Vert q_{n+1}-u^{\ast }\right\Vert  &=&\left\Vert \left(
1-a_{n}^{0}\right) r_{n}+a_{n}^{0}Tr_{n}-u^{\ast }\right\Vert   \label{eqn44}
\\
&\geq &\left( 1-a_{n}^{0}\right) \left\Vert r_{n}-u^{\ast }\right\Vert
-a_{n}^{0}\left\Vert Tr_{n}-Tu^{\ast }\right\Vert   \notag \\
&\geq &\left[ 1-a_{n}^{0}\left( 1+\delta \right) \right] \left\Vert
r_{n}-u^{\ast }\right\Vert   \notag \\
&\geq &\left[ 1-a_{n}^{0}\left( 1+\delta \right) \right] \left\{ \left(
1-a_{n}^{1}\right) \left\Vert s_{n}-u^{\ast }\right\Vert -a_{n}^{1}\delta
\left\Vert s_{n}-u^{\ast }\right\Vert \right\}   \notag \\
&=&\left[ 1-a_{n}^{0}\left( 1+\delta \right) \right] \left[
1-a_{n}^{1}\left( 1+\delta \right) \right] \left\Vert s_{n}-u^{\ast
}\right\Vert   \notag \\
&\geq &\left[ 1-a_{n}^{0}\left( 1+\delta \right) \right] \left[
1-a_{n}^{1}\left( 1+\delta \right) \right] \left\{ \left( 1-a_{n}^{2}\right)
\left\Vert q_{n}-u^{\ast }\right\Vert -a_{n}^{2}\delta \left\Vert
q_{n}-u^{\ast }\right\Vert \right\}   \notag \\
&=&\left[ 1-a_{n}^{0}\left( 1+\delta \right) \right] \left[
1-a_{n}^{1}\left( 1+\delta \right) \right] \left[ 1-a_{n}^{2}\left( 1+\delta
\right) \right] \left\Vert q_{n}-u^{\ast }\right\Vert   \notag \\
&\geq &\cdots   \notag \\
&\geq &\left\Vert q_{0}-u^{\ast }\right\Vert ^{n+1}\prod\limits_{k=0}^{n}
\left[ 1-a_{k}^{0}\left( 1+\delta \right) \right] \left[ 1-a_{k}^{1}\left(
1+\delta \right) \right] \left[ 1-a_{k}^{2}\left( 1+\delta \right) \right] 
\text{.}  \notag
\end{eqnarray}%
Using now (42) and (43)%
\begin{equation}
\frac{\left\Vert x_{n+1}-u^{\ast }\right\Vert }{\left\Vert q_{n+1}-u^{\ast
}\right\Vert }\leq \frac{\delta ^{2\left( n+1\right) }\left\Vert
x_{0}-u^{\ast }\right\Vert ^{n+1}\dprod\nolimits_{k=0}^{n}\left[
1-a_{k}^{1}a_{k}^{2}\left( 1-\delta \right) \right] }{\left\Vert
q_{0}-u^{\ast }\right\Vert ^{n+1}\prod\limits_{k=0}^{n}\left[
1-a_{k}^{0}\left( 1+\delta \right) \right] \left[ 1-a_{k}^{1}\left( 1+\delta
\right) \right] \left[ 1-a_{k}^{2}\left( 1+\delta \right) \right] }\text{.}
\label{eqn45}
\end{equation}%
Define 
\begin{equation}
\theta _{n}=\frac{\delta ^{2\left( n+1\right) }\dprod\nolimits_{k=0}^{n}%
\left[ 1-a_{k}^{1}a_{k}^{2}\left( 1-\delta \right) \right] }{%
\dprod\nolimits_{k=0}^{n}\left[ 1-a_{k}^{0}\left( 1+\delta \right) \right] %
\left[ 1-a_{k}^{1}\left( 1+\delta \right) \right] \left[ 1-a_{k}^{2}\left(
1+\delta \right) \right] }\text{.}  \label{eqn46}
\end{equation}%
By the assumtion%
\begin{eqnarray}
&&\lim_{n\rightarrow \infty }\frac{\theta _{n+1}}{\theta _{n}}  \label{eqn47}
\\
&=&\lim_{n\rightarrow \infty }\frac{\frac{\delta ^{2\left( n+2\right)
}\dprod\nolimits_{k=0}^{n+1}\left[ 1-a_{k}^{1}a_{k}^{2}\left( 1-\delta
\right) \right] }{\dprod\nolimits_{k=0}^{n+1}\left[ 1-a_{k}^{0}\left(
1+\delta \right) \right] \left[ 1-a_{k}^{1}\left( 1+\delta \right) \right] %
\left[ 1-a_{k}^{2}\left( 1+\delta \right) \right] }}{\frac{\delta ^{2\left(
n+1\right) }\dprod\nolimits_{k=0}^{n}\left[ 1-a_{k}^{1}a_{k}^{2}\left(
1-\delta \right) \right] }{\dprod\nolimits_{k=0}^{n}\left[ 1-a_{k}^{0}\left(
1+\delta \right) \right] \left[ 1-a_{k}^{1}\left( 1+\delta \right) \right] %
\left[ 1-a_{k}^{2}\left( 1+\delta \right) \right] }}  \notag \\
&=&\lim_{n\rightarrow \infty }\frac{\delta ^{2}\left[
1-a_{n+1}^{1}a_{n+1}^{2}\left( 1-\delta \right) \right] }{\left[
1-a_{n+1}^{0}\left( 1+\delta \right) \right] \left[ 1-a_{n+1}^{1}\left(
1+\delta \right) \right] \left[ 1-a_{n+1}^{2}\left( 1+\delta \right) \right] 
}  \notag \\
&=&\delta ^{2}<1\text{.}  \notag
\end{eqnarray}%
It thus follows from ratio test that $\sum\limits_{n=0}^{\infty }\theta
_{n}<\infty $. Hence, we have $\lim_{n\rightarrow \infty }\theta _{n}=0$
which implies that the iterative sequence defined by (7) converges faster
than the iterative sequence defined by SP iteration method (6).

Using Noor iteration method (5), we get%
\begin{eqnarray}
\left\Vert \omega _{n+1}-u^{\ast }\right\Vert  &=&\left\Vert \left(
1-a_{n}^{0}\right) \omega _{n}+a_{n}^{0}T\varpi _{n}-u^{\ast }\right\Vert 
\label{eqn48} \\
&\geq &\left( 1-a_{n}^{0}\right) \left\Vert \omega _{n}-u^{\ast }\right\Vert
-a_{n}^{0}\left\Vert T\varpi _{n}-Tu^{\ast }\right\Vert   \notag \\
&\geq &\left( 1-a_{n}^{0}\right) \left\Vert \omega _{n}-u^{\ast }\right\Vert
-a_{n}^{0}\delta \left\Vert \varpi _{n}-u^{\ast }\right\Vert   \notag \\
&\geq &\left[ 1-a_{n}^{0}-a_{n}^{0}\delta \left( 1-a_{n}^{1}\right) \right]
\left\Vert \omega _{n}-u^{\ast }\right\Vert -a_{n}^{0}a_{n}^{1}\delta
^{2}\left\Vert \rho _{n}-u^{\ast }\right\Vert   \notag \\
&\geq &\left\{ 1-a_{n}^{0}-a_{n}^{0}\delta \left( 1-a_{n}^{1}\right)
-a_{n}^{0}a_{n}^{1}\delta ^{2}\left[ 1-a_{n}^{2}\left( 1-\delta \right) %
\right] \right\} \left\Vert \omega _{n}-u^{\ast }\right\Vert   \notag \\
&\geq &\left\{ 1-a_{n}^{0}-a_{n}^{0}\delta \left[ 1-a_{n}^{1}\left( 1-\delta
\right) \right] \right\} \left\Vert \omega _{n}-u^{\ast }\right\Vert   \notag
\\
&\geq &\left[ 1-a_{n}^{0}\left( 1+\delta \right) \right] \left\Vert \omega
_{n}-u^{\ast }\right\Vert   \notag \\
&\geq &\cdots   \notag \\
&\geq &\left\Vert \omega _{0}-u^{\ast }\right\Vert
^{n+1}\prod\limits_{k=0}^{n}\left[ 1-a_{k}^{0}\left( 1+\delta \right) \right]
\text{.}  \notag
\end{eqnarray}%
It follows by (42) and (47) that%
\begin{equation}
\frac{\left\Vert x_{n+1}-u^{\ast }\right\Vert }{\left\Vert \omega
_{n+1}-x_{\ast }\right\Vert }\leq \frac{\delta ^{2\left( n+1\right)
}\left\Vert x_{0}-u^{\ast }\right\Vert ^{n+1}\dprod\nolimits_{k=0}^{n}\left[
1-a_{k}^{1}a_{k}^{2}\left( 1-\delta \right) \right] }{\left\Vert \omega
_{0}-u^{\ast }\right\Vert ^{n+1}\prod\limits_{k=0}^{n}\left[
1-a_{k}^{0}\left( 1+\delta \right) \right] }\text{.}  \label{eqn49}
\end{equation}%
Define 
\begin{equation}
\theta _{n}=\frac{\delta ^{2\left( n+1\right) }\dprod\nolimits_{k=0}^{n}%
\left[ 1-a_{k}^{1}a_{k}^{2}\left( 1-\delta \right) \right] }{%
\prod\limits_{k=0}^{n}\left[ 1-a_{k}^{0}\left( 1+\delta \right) \right] }%
\text{.}  \label{eqn50}
\end{equation}%
By the assumption%
\begin{eqnarray}
\lim_{n\rightarrow \infty }\frac{\theta _{n+1}}{\theta _{n}}
&=&\lim_{n\rightarrow \infty }\frac{\delta ^{2}\left[
1-a_{n+1}^{1}a_{n+1}^{2}\left( 1-\delta \right) \right] }{\left[
1-a_{n+1}^{0}\left( 1+\delta \right) \right] }  \label{eqn51} \\
&=&\delta ^{2}<1\text{.}  \notag
\end{eqnarray}%
It thus follows from ratio test that $\sum\limits_{n=0}^{\infty }\theta
_{n}<\infty $. Hence, we have $\lim_{n\rightarrow \infty }\theta _{n}=0$
which implies that the iterative sequence defined by (7) converges faster
than the iterative sequence defined by Noor iteration method (5).
\end{proof}

By use of the following example due to \cite{QR}, it was shown in (\cite{CR}%
, Example 4.1) that CR iterative method \cite{CR} is faster than all of \
Picard (3), S \cite{S}, Noor (5) and SP (6) iterative methods for a
particular class of \ operators which is included in the class of
weak-contraction mappings satisfying (11). In the following, for the sake of
consistent comparison, we will use the same example as that of \ (\cite{CR},
Example 4.1) in order to compare the rates of convergence between the
Picard-S iterative scheme (7) and the CR iteration method \cite{CR} for
weak-contraction mappings. In the following example, for convenience, we use
the notations $\left( PS_{n}\right) $ and $\left( CR_{n}\right) $ for the
iterative sequences associated to Picard-S (7) and CR \cite{CR} iteration
methods, respectively.

\begin{example}
\cite{QR} Define a mapping $T:\left[ 0,1\right] \rightarrow \left[ 0,1\right]
$ as $Tx=\frac{x}{2}$. Let $a_{n}^{0}=a_{n}^{1}=a_{n}^{2}=0$, for $%
n=1,2,...,24$ and $a_{n}^{0}=a_{n}^{1}=a_{n}^{2}=\frac{4}{\sqrt{n}}$, for
all $n\geq 25$.

It can be seen easily that the mapping $T$ satisfies condition (11) with the
unique fixed point $0\in F_{T}$. Furthermore, it is easy to see that Example
1 satisfies all the conditions of Theorem 1. Indeed, let $x_{0}\neq 0$ be an
initial point for the iterative sequences $\left( PS_{n}\right) $ and $%
\left( CR_{n}\right) $. Utilizing Picard-S (7) and CR \cite{CR} iteration
methods we obtain 
\begin{eqnarray}
PS_{n} &=&\frac{1}{2}\left( \frac{1}{2}-\frac{4}{n}\right) x_{n}  \notag \\
&=&\cdots   \notag \\
&=&\dprod\limits_{k=25}^{n}\left( \frac{1}{4}-\frac{2}{k}\right) x_{0}\text{,%
}  \label{eqn52}
\end{eqnarray}%
\begin{eqnarray}
CR_{n} &=&\left( \frac{1}{2}-\frac{1}{\sqrt{n}}-\frac{4}{n}+\frac{8}{n\sqrt{n%
}}\right) x_{n}  \notag \\
&=&\cdots   \notag \\
&=&\dprod\limits_{k=25}^{n}\left( \frac{1}{2}-\frac{1}{\sqrt{k}}-\frac{4}{k}+%
\frac{8}{k\sqrt{k}}\right) x_{0}\text{.}  \label{eqn53}
\end{eqnarray}%
It follows from (51) and (52) that%
\begin{equation*}
\frac{\left\vert PS_{n}-0\right\vert }{\left\vert CR_{n}-0\right\vert }=%
\frac{\dprod\limits_{k=25}^{n}\left( \frac{1}{4}-\frac{2}{k}\right) x_{0}}{%
\dprod\limits_{k=25}^{n}\left( \frac{1}{2}-\frac{1}{\sqrt{k}}-\frac{4}{k}+%
\frac{8}{k\sqrt{k}}\right) x_{0}}
\end{equation*}%
\begin{equation*}
\text{ \ \ \ \ \ \ \ }=\dprod\limits_{k=25}^{n}\frac{\frac{1}{4}-\frac{2}{k}%
}{\frac{1}{2}-\frac{1}{\sqrt{k}}-\frac{4}{k}+\frac{8}{k\sqrt{k}}}
\end{equation*}%
\begin{equation*}
\text{ \ \ \ \ \ \ \ \ \ \ \ \ \ \ \ \ \ \ }=\dprod\limits_{k=25}^{n}\frac{%
\left( k-8\right) \sqrt{k}}{2\left( k\sqrt{k}-2k-8\sqrt{k}+16\right) }
\end{equation*}%
\begin{eqnarray}
\text{ \ \ \ \ \ \ } &=&\dprod\limits_{k=25}^{n}\frac{\left( k-8\right) 
\sqrt{k}}{2\left( \sqrt{k}-2\right) \left( k-8\right) }  \notag \\
&=&\dprod\limits_{k=25}^{n}\frac{\sqrt{k}}{2\left( \sqrt{k}-2\right) }\text{.%
}  \label{eqn54}
\end{eqnarray}%
For all $k\geq 25$, we have%
\begin{eqnarray}
\frac{\left( k-2\right) \left( \sqrt{k}-4\right) }{4} &>&1  \notag \\
&\Rightarrow &\left( k-2\right) \left( \sqrt{k}-4\right) >4  \notag \\
&\Rightarrow &k\left( \sqrt{k}-4\right) >2\left( \sqrt{k}-2\right)   \notag
\\
&\Rightarrow &\frac{\sqrt{k}-4}{2\left( \sqrt{k}-2\right) }>\frac{1}{k} 
\notag \\
&\Rightarrow &\frac{\sqrt{k}}{2\left( \sqrt{k}-2\right) }<1-\frac{1}{k}\text{%
,}  \label{eqn55}
\end{eqnarray}%
which yields%
\begin{equation}
\frac{\left\vert PS_{n}-0\right\vert }{\left\vert CR_{n}-0\right\vert }%
=\dprod\limits_{k=25}^{n}\frac{\sqrt{k}}{2\left( \sqrt{k}-2\right) }%
<\dprod\limits_{k=25}^{n}\left( 1-\frac{1}{k}\right) =\frac{24}{n}\text{.}
\label{eqn56}
\end{equation}%
Therefore, we have%
\begin{equation}
\lim_{n\rightarrow \infty }\frac{\left\vert PS_{n}-0\right\vert }{\left\vert
CR_{n}-0\right\vert }=0\text{,}  \label{eqn57}
\end{equation}%
which implies that the Picard-S iterative scheme (7) is faster than the CR
iteration method \cite{CR}$.$
\end{example}

Having regard to R. Chugh et al.'s result (\cite{CR}, Example 4.1), L.B.
Ciric et al.'s results \cite{Ciric} and Example 1 above, we conclude that
Picard-S iteration method is faster than all Picard (3), Mann \cite{Mann},
Ishikawa \cite{Ishikawa}, S \cite{S}, Noor (5) and SP (6) iterative methods.

We are now able to establish the following data dependence result.

\begin{theorem}
Let $T$ with fixed point $u^{\ast }\in F_{T}\neq \emptyset $ be as in
Theorem 1 and $\widetilde{T}$ an approximate operator of $T$. Let $\left\{
x_{n}\right\} _{n=0}^{\infty }$ be an iterative sequence generated by (7)
for $T$ and define an iterative sequence $\left\{ \widetilde{x}_{n}\right\}
_{n=0}^{\infty }$ as follows%
\begin{equation}
\left\{ 
\begin{array}{c}
\widetilde{x}_{0}\in D\text{, \ \ \ \ \ \ \ \ \ \ \ \ \ \ \ \ \ \ \ \ \ \ \
\ \ \ \ \ \ \ \ \ \ \ \ \ \ \ \ \ } \\ 
\widetilde{x}_{n+1}=\widetilde{T}\widetilde{y}_{n}\text{, \ \ \ \ \ \ \ \ \
\ \ \ \ \ \ \ \ \ \ \ \ \ \ \ \ \ \ \ \ \ \ \ \ \ } \\ 
\widetilde{y}_{n}=\left( 1-a_{n}^{1}\right) \widetilde{T}\widetilde{x}%
_{n}+a_{n}^{1}\widetilde{T}\widetilde{z}_{n}\text{, \ \ \ \ \ \ } \\ 
\widetilde{z}_{n}=\left( 1-a_{n}^{2}\right) \widetilde{x}_{n}+a_{n}^{2}%
\widetilde{T}\widetilde{x}_{n}\text{, }n\in 
\mathbb{N}
\text{,}%
\end{array}%
\right.   \label{eqn58}
\end{equation}%
where $\left\{ a_{n}^{i}\right\} _{n=0}^{\infty }$, $i\in \left\{
1,2\right\} $ be real sequences in $\left[ 0,1\right] $ satisfying (i) $%
\frac{1}{2}\leq a_{n}^{1}a_{n}^{2}$ for all $n\in 
\mathbb{N}
$, and (ii) $\sum\limits_{n=0}^{\infty }a_{n}^{1}a_{n}^{2}=\infty $. If $%
\widetilde{T}\widetilde{u}^{\ast }=\widetilde{u}^{\ast }$ such that $%
\widetilde{x}_{n}\rightarrow \widetilde{u}^{\ast }$ as $n\rightarrow \infty $%
, then we have%
\begin{equation}
\left\Vert u^{\ast }-\widetilde{u}^{\ast }\right\Vert \leq \frac{%
5\varepsilon }{1-\delta }\text{,}  \label{eqn59}
\end{equation}%
where $\varepsilon >0$ is a fixed number.
\end{theorem}

\begin{proof}
It follows from (7), (11), (12), and (57) that%
\begin{eqnarray}
\left\Vert z_{n}-\widetilde{z}_{n}\right\Vert &\leq &\left(
1-a_{n}^{2}\right) \left\Vert x_{n}-\widetilde{x}_{n}\right\Vert
+a_{n}^{2}\left\Vert Tx_{n}-\widetilde{T}\widetilde{x}_{n}\right\Vert  \notag
\\
&\leq &\left[ 1-a_{n}^{2}+a_{n}^{2}\delta \right] \left\Vert x_{n}-%
\widetilde{x}_{n}\right\Vert +a_{n}^{2}L\left\Vert x_{n}-Tx_{n}\right\Vert
+a_{n}^{2}\varepsilon \text{,}  \label{eqn60}
\end{eqnarray}%
\begin{eqnarray}
\left\Vert y_{n}-\widetilde{y}_{n}\right\Vert &\leq &\left(
1-a_{n}^{1}\right) \delta \left\Vert x_{n}-\widetilde{x}_{n}\right\Vert
+a_{n}^{1}\delta \left\Vert z_{n}-\widetilde{z}_{n}\right\Vert  \notag \\
&&+\left( 1-a_{n}^{1}\right) L\left\Vert x_{n}-Tx_{n}\right\Vert
+a_{n}^{1}L\left\Vert z_{n}-Tz_{n}\right\Vert  \notag \\
&&+\left( 1-a_{n}^{1}\right) \varepsilon +a_{n}^{1}\varepsilon \text{,}
\label{eqn61}
\end{eqnarray}%
\begin{equation}
\left\Vert x_{n+1}-\widetilde{x}_{n+1}\right\Vert \leq \delta \left\Vert
y_{n}-\widetilde{y}_{n}\right\Vert +L\left\Vert y_{n}-Ty_{n}\right\Vert
+\varepsilon \text{.}  \label{eqn62}
\end{equation}%
From the relations (59), (60), and (61)%
\begin{eqnarray}
\left\Vert x_{n+1}-\widetilde{x}_{n+1}\right\Vert &\leq &\delta ^{2}\left[
1-a_{n}^{1}a_{n}^{2}\left( 1-\delta \right) \right] \left\Vert x_{n}-%
\widetilde{x}_{n}\right\Vert  \notag \\
&&+\left\{ a_{n}^{1}a_{n}^{2}\delta ^{2}L+\left( 1-a_{n}^{1}\right) \delta
L\right\} \left\Vert x_{n}-Tx_{n}\right\Vert  \notag \\
&&+L\left\Vert y_{n}-Ty_{n}\right\Vert +a_{n}^{1}\delta L\left\Vert
z_{n}-Tz_{n}\right\Vert  \notag \\
&&+a_{n}^{1}a_{n}^{2}\delta ^{2}\varepsilon +\left( 1-a_{n}^{1}\right)
\delta \varepsilon +a_{n}^{1}\delta \varepsilon +\varepsilon \text{.}
\label{eqn63}
\end{eqnarray}%
Since $a_{n}^{1}$, $a_{n}^{2}\in \left[ 0,1\right] $ $\ $and $\frac{1}{2}%
\leq a_{n}^{1}a_{n}^{2}$ for all $n\in 
\mathbb{N}
$%
\begin{equation}
1-a_{n}^{1}a_{n}^{2}\leq a_{n}^{1}a_{n}^{2}\text{,}  \label{eqn64}
\end{equation}%
\begin{equation}
1-a_{n}^{1}\leq 1-a_{n}^{1}a_{n}^{2}\leq a_{n}^{1}a_{n}^{2}\text{,}
\label{eqn65}
\end{equation}%
\begin{equation}
1\leq 2a_{n}^{1}a_{n}^{2}\text{.}  \label{eqn66}
\end{equation}%
Use of the facts $\delta $, $\delta ^{2}\in \left( 0,1\right) $, (63), (64),
and (65) in (62) yields 
\begin{eqnarray}
\left\Vert x_{n+1}-\widetilde{x}_{n+1}\right\Vert &\leq &\left[
1-a_{n}^{1}a_{n}^{2}\left( 1-\delta \right) \right] \left\Vert x_{n}-%
\widetilde{x}_{n}\right\Vert  \notag \\
&&+a_{n}^{1}a_{n}^{2}\left( 1-\delta \right) \left\{ \frac{L\delta \left(
1+\delta \right) \left\Vert x_{n}-Tx_{n}\right\Vert }{1-\delta }\right. 
\notag \\
&&\left. +\frac{2L\left\Vert y_{n}-Ty_{n}\right\Vert +2\delta L\left\Vert
z_{n}-Tz_{n}\right\Vert +5\varepsilon }{1-\delta }\right\} \text{.}
\label{eqn67}
\end{eqnarray}%
Define%
\begin{eqnarray}
\beta _{n} &:&=\left\Vert x_{n}-\widetilde{x}_{n}\right\Vert \text{,}
\label{eqn68} \\
\mu _{n} &:&=a_{n}^{1}a_{n}^{2}\left( 1-\delta \right) \in \left( 0,1\right) 
\text{,}  \notag \\
\gamma _{n} &:&=\frac{L\delta \left( 1+\delta \right) \left\Vert
x_{n}-Tx_{n}\right\Vert +2L\left\Vert y_{n}-Ty_{n}\right\Vert +2\delta
L\left\Vert z_{n}-Tz_{n}\right\Vert +5\varepsilon }{1-\delta }\geq 0\text{.}
\notag
\end{eqnarray}%
Hence, the inequality (66) perform all assumptions in Lemma 2 and thus an
application of Lemma 2 to (66) yields 
\begin{eqnarray}
0 &\leq &\lim \sup_{n\rightarrow \infty }\left\Vert x_{n}-\widetilde{x}%
_{n}\right\Vert  \label{eqn69} \\
&\leq &\lim \sup_{n\rightarrow \infty }\frac{L\delta \left( 1+\delta \right)
\left\Vert x_{n}-Tx_{n}\right\Vert +2L\left\Vert y_{n}-Ty_{n}\right\Vert
+2\delta L\left\Vert z_{n}-Tz_{n}\right\Vert +5\varepsilon }{1-\delta }. 
\notag
\end{eqnarray}%
We know from Theorem 1 that $\lim_{n\rightarrow \infty }x_{n}=u^{\ast }$ and
since $Tu^{\ast }=u^{\ast }$%
\begin{equation}
\lim_{n\rightarrow \infty }\left\Vert x_{n}-Tx_{n}\right\Vert
=\lim_{n\rightarrow \infty }\left\Vert y_{n}-Ty_{n}\right\Vert
=\lim_{n\rightarrow \infty }\left\Vert z_{n}-Tz_{n}\right\Vert =0\text{.}
\label{eqn70}
\end{equation}%
Therefore the inequality (68) becomes 
\begin{equation}
\left\Vert u^{\ast }-\widetilde{u}^{\ast }\right\Vert \leq \frac{%
5\varepsilon }{1-\delta }.  \label{eqn71}
\end{equation}
\end{proof}

\end{document}